\documentclass{amsart}
\usepackage{amssymb, amsmath, amsthm, graphics, comment, xspace, enumerate}
\newtheorem{thm}{Theorem}[section]

\newtheorem{lem}[thm]{Lemma}
\newtheorem{prop}[thm]{Proposition}
\theoremstyle{definition}
\newtheorem{defn}[thm]{Definition}
\newtheorem{example}[thm]{Example}
\theoremstyle{remark}
\newtheorem{rem}[thm]{Remark}
\numberwithin{equation}{section}

\begin{document}
\title[Generalized almost periodic and generalized asymptotically...]{Generalized almost periodic and generalized asymptotically almost periodic type functions in Lebesgue spaces with variable exponents
$L^{p(x)}$}

\author{Toka Diagana}
\address{Department of Mathematics, Howard University, 2441 6th Street NW, Washington, DC
20059, USA}
\email{tokadiag@gmail.com}

\author{Marko Kosti\' c}
\address{Faculty of Technical Sciences,
University of Novi Sad,
Trg D. Obradovi\' ca 6, 21125 Novi Sad, Serbia}
\email{marco.s@verat.net}

{\renewcommand{\thefootnote}{} \footnote{
\\ \text{  }  \ \    {\it Key words and phrases.} Lebesgue spaces with variable exponents,
generalized almost periodicity with variable exponents, generalized asymptotical almost periodicity with variable exponents, abstract Volterra integro-differential equations, abstract fractional differential equations.
\\  \text{  }  \ \ Please, insert the number of your grant here... The second named author is partially supported by grant 174024 of Ministry
of Science and Technological Development, Republic of Serbia.}}

\begin{abstract}
In the paper under review, we introduce the notions of various types of generalized (asymptotical) almost periodicity with variable exponents. 
We define and thoroughly analyze an important subclass of (asymptotically) Stepanov almost periodic functions which contains all (asymptotically) almost periodic functions.
We provide a great number of relevant applications to abstract Volterra integro-differential equations in Banach spaces.
\end{abstract}
\maketitle

\section{Introduction and Preliminaries}\label{intro1}

As mentioned in the abstract, the main aim of this paper is to introduce the notions of various types of generalized (asymptotical) almost periodicity with variable exponents. For a given measurable function $p : [0,1] \rightarrow [1,\infty],$ $p\in {\mathcal P}([0,1])$ shortly, we define the notion of an $S^{p(x)}$-almost periodic function and further analyze this concept. For $p(x)\equiv p \geq 1,$ the introduced notion is equivalent to the usually considered notion of $S^{p}$-almost periodicity. We prove that any almost periodic function has to be $S^{p(x)}$-almost periodic as well as that any $S^{p(x)}$-almost periodic function has to be Stepanov $1$-almost periodic ($p\in {\mathcal P}([0,1])$); similar assertions hold for the asymptotical $S^{p(x)}$-almost periodicity...

The organization and main ideas of this paper are briefly described as follows...

We use the standard notation throughout the paper.
Unless specifed otherwise,
we assume
that $(X,\| \cdot \|)$ is a complex Banach space. If $Y$ is also such a space, then we denote by
$L(X,Y)$ the space of all continuous linear mappings from $X$ into
$Y;$ $L(X)\equiv L(X,X).$ Assuming $A$ is a closed linear operator
acting on $X,$
then the domain, kernel space and range of $A$ will be denoted by
$D(A),$ $N(A)$ and $R(A),$
respectively. Since no confusion
seems likely, we will identify $A$ with its graph. By $[D(A)]$ we designate the Banach space 
$D(A)$ equipped with the graph norm $\|x\|_{[D(A)]}:=\|x\|+\|Ax\|,$ $x\in D(A).$

If the numbers $s\in {\mathbb R}$ and $\theta \in (0,\pi]$ are given in advance, then we set $\lceil s \rceil:=\inf \{
l\in {\mathbb Z} : s\leq l \}$ and $\Sigma_{\theta}:=\{ z\in {\mathbb C} \setminus \{0\} :
|\arg (z)|<\theta \}.$ 
The symbol $C(I: X),$ where $I={\mathbb R}$ or $I=[0,\infty),$ stands for the space consisting of all $X$-valued
continuous functions on the interval $I$. By $C_{b}(I: X)$ and $BUC(I: X)$ we denote the closed subspaces of $C(I: X)$ consisting of all bounded and bounded uniformly continuous functions, respectively. Any of this spaces is a Banach one equipped with the sup-norm.
The Gamma function is denoted by
$\Gamma(\cdot)$ and the principal branch is always used to take
the powers; the convolution like
mapping $\ast$ is given by $f\ast g(t):=\int_{0}^{t}f(t-s)g(s)\,
ds .$ Set $g_{\zeta}(t):=t^{\zeta-1}/\Gamma(\zeta),$ $\zeta>0.$

The first conference on fractional calculus and fractional differential equations was held in New Haven (1974).  Since then, fractional calculus has gained more
and more attention due to its wide applications in various fields of science, such as mathematical physics,
engineering, biology, aerodynamics, chemistry, economics etc. Fairly complete information about fractional calculus and fractional
differential equations can be obtained
by consulting
\cite{bajlekova}, \cite{Diet}, \cite{kilbas}, \cite{knjigaho} and references cited therein.
The
Mittag-Leffler function $E_{\alpha,\beta}(z),$ defined by
$$
E_{\alpha,\beta}(z):=\sum_{n=0}^{\infty}\frac{z^{n}}{\Gamma(\alpha
n+\beta)},\quad z\in {\mathbb C},
$$
is known to play a crucial role in the analysis of fractional
differential equations. Set, for short,
$
E_{\alpha}(z):=E_{\alpha,1}(z),$ $ z\in
{\mathbb C}.
$
Assume that $\gamma \in (0,1).$
The Wright function\index{function!Wright}
$\Phi_{\gamma}(\cdot)$ is defined by
$$
\Phi_{\gamma}(t):={{\mathcal
L}^{-1}}\bigl(E_{\gamma}(-\lambda)\bigr)(t),\quad t\geq 0,
$$
where ${\mathcal L}^{-1}$ denotes the inverse Laplace transform.
The Wright function $\Phi_{\gamma}(\cdot)$
is an entire function which can be equivalently introduced by the formula
$$
\Phi_{\gamma}(z)=\sum \limits_{n=0}^{\infty}
\frac{(-z)^{n}}{n! \Gamma (1-\gamma -\gamma n)},\quad z\in {\mathbb C}.
$$

Let $\gamma \in (0,1).$  Then
the Caputo fractional derivative\index{fractional derivatives!Caputo}
${\mathbf D}_{t}^{\gamma}u(t)$ is defined for those functions
$u :[0,\infty) \rightarrow X$ satisfying that, for every $T>0,$ we have $u _{| (0,T]}(\cdot) \in C((0,T]: X),$ $u(\cdot)-u(0) \in L^{1}((0,T) : X)$
and $g_{1-\gamma}\ast (u(\cdot)-u(0)) \in W^{1,1}((0,T) : X),$
by
$$
{\mathbf
D}_{t}^{\gamma}u(t)=\frac{d}{dt}\Biggl[g_{1-\gamma}
\ast \Bigl(u(\cdot)-u(0)\Bigr)\Biggr](t),\quad t\in (0,T];
$$
see \cite[p. 7]{bajlekova} for the notion of Sobolev space $W^{1,1}((0,T) : X).$
The Weyl-Liouville fractional derivative
$D_{t,+}^{\gamma}u(t)$ of order $\gamma$ is defined for those continuous functions
$u : {\mathbb R} \rightarrow X$
such that $t\mapsto \int_{-\infty}^{t}g_{1-\gamma}(t-s)u(s)\, ds,$ $t\in {\mathbb R}$ is a well-defined continuously differentiable mapping, by
$$
D_{t,+}^{\gamma}u(t):=\frac{d}{dt}\int_{-\infty}^{t}g_{1-\gamma}(t-s)u(s)\, ds,\quad t\in {\mathbb R}.
$$
Set $
{\mathbf D}_{t}^{1}u(t):=(d/dt)u(t)$ and $
D_{t,+}^{1}u(t):=-(d/dt)u(t).$

\subsection{Multivalued
linear operators and degenerate resolvent operator families}\label{mlos} 

Suppose that $X$ and $Y$ are two Banach spaces.
A multivalued map (multimap) ${\mathcal A} : X \rightarrow P(Y)$ is said to be a multivalued
linear operator, MLO for short, iff the following holds:
\begin{itemize}
\item[(i)] $D({\mathcal A}) := \{x \in X : {\mathcal A}x \neq \emptyset\}$ is a linear subspace of $X$;
\item[(ii)] ${\mathcal A}x +{\mathcal A}y \subseteq {\mathcal A}(x + y),$ $x,\ y \in D({\mathcal A})$
and $\lambda {\mathcal A}x \subseteq {\mathcal A}(\lambda x),$ $\lambda \in {\mathbb C},$ $x \in D({\mathcal A}).$
\end{itemize}
In the case that $X=Y,$ then we say that ${\mathcal A}$ is an MLO in $X.$
It is well known that
for any $x,\ y\in D({\mathcal A})$ and $\lambda,\ \eta \in {\mathbb C}$ with $|\lambda| + |\eta| \neq 0,$ we
have $\lambda {\mathcal A}x + \eta {\mathcal A}y = {\mathcal A}(\lambda x + \eta y).$ If ${\mathcal A}$ is an MLO, then ${\mathcal A}0$ is a linear manifold in $Y$
and ${\mathcal A}x = f + {\mathcal A}0$ for any $x \in D({\mathcal A})$ and $f \in {\mathcal A}x.$ Define 
the kernel space $N({\mathcal A})$ of ${\mathcal A}$ and the range $R({\mathcal A})$ of ${\mathcal A}$ 
by $N({\mathcal A}):= \{x \in D({\mathcal A}) : 0 \in {\mathcal A}x\}$ and
$R({\mathcal A}):=\{{\mathcal A}x :  x\in D({\mathcal A})\},$ respectively. We write ${\mathcal A} \subseteq {\mathcal B}$ iff $D({\mathcal A}) \subseteq D({\mathcal B})$ and ${\mathcal A}x \subseteq {\mathcal B}x$
for all $x\in D({\mathcal A}).$

Sums, mutual products, taking powers and products with complex scalars are standard operations for multivalued
linear operators (see e.g. \cite{cross}, \cite{faviniyagi} and \cite{FKP}). It is said that an MLO operator  ${\mathcal A} : X\rightarrow P(Y)$ is closed iff for any
sequences $(x_{n})$ in $D({\mathcal A})$ and $(y_{n})$ in $Y$ such that $y_{n}\in {\mathcal A}x_{n}$ for all $n\in {\mathbb N}$ we have that the suppositions $\lim_{n \rightarrow \infty}x_{n}=x$ and
$\lim_{n \rightarrow \infty}y_{n}=y$ imply
$x\in D({\mathcal A})$ and $y\in {\mathcal A}x.$ 

Concerning the $C$-resolvent sets of MLOs in Banach spaces, our standing hyportheses will be that ${\mathcal A}$ is an MLO in $X$, as well as that $C\in L(X)$ is injective and $C{\mathcal A}\subseteq {\mathcal A}C.$
The $C$-resolvent set of ${\mathcal A},$ $\rho_{C}({\mathcal A})$ for short, is defined as the union of those complex numbers
$\lambda \in {\mathbb C}$ for which
\begin{itemize}
\item[(i)] $R(C)\subseteq R(\lambda-{\mathcal A})$;
\item[(ii)] $(\lambda - {\mathcal A})^{-1}C$ is a single-valued linear continuous operator on $X.$
\end{itemize}
The operator $\lambda \mapsto (\lambda -{\mathcal A})^{-1}C$ is called the $C$-resolvent of ${\mathcal A}$ ($\lambda \in \rho_{C}({\mathcal A})$); the resolvent set of ${\mathcal A}$ is then defined by $\rho({\mathcal A}):=\rho_{I}({\mathcal A}),$ where $I$ denotes the identity operator on $X.$ Set $R(\lambda : {\mathcal A})\equiv  (\lambda -{\mathcal A})^{-1}$  ($\lambda \in \rho({\mathcal A})$).
The basic properties of $C$-resolvent sets of single-valued linear operators continue to hold in our framework (cf. \cite{FKP} for more details).  For instance, $\rho({\mathcal A})$ is always an open subset of ${\mathbb C}$ and $\rho({\mathcal A})\neq \emptyset$ implies that ${\mathcal A}$ is closed.

In the sequel, we will employ the following important condition:

\begin{itemize}
\item[(P)]
There exist finite constants $c,\ M>0$ and $\beta \in (0,1]$ such that
$$
\Psi:=\Bigl\{ \lambda \in {\mathbb C} : \Re \lambda \geq -c\bigl( |\Im \lambda| +1 \bigr) \Bigr\} \subseteq \rho({\mathcal A})
$$
and
$$
\| R(\lambda : {\mathcal A})\| \leq M\bigl( 1+|\lambda|\bigr)^{-\beta},\quad \lambda \in \Psi.
$$
\end{itemize}

Then degenerate strongly continuous semigroup $(T(t))_{t> 0}\subseteq L(X)$ generated by ${\mathcal A}$ satisfies estimate
$\|T(t) \| \leq M e^{-ct}t^{\beta -1},$ $t> 0$ for some finite constant $M>0$. Furthermore, we know that $(T(t))_{t> 0}$ is given by the formula
$$
T(t)x=\frac{1}{2\pi i}\int_{\Gamma}e^{\lambda t}\bigl(  \lambda -{\mathcal A} \bigr)^{-1}x\, d\lambda,\quad t>0,\ x\in X,
$$
where $\Gamma$ is the upwards oriented curve $\lambda=-c(|\eta|+1)+i\eta$ ($\eta \in {\mathbb R}$).
Assume that $0<\gamma <1$ and $\nu >-\beta .$ Set
\begin{align*}
T_{\gamma,\nu}(t)x:=t^{\gamma \nu}\int^{\infty}_{0}s^{\nu}\Phi_{\gamma}( s)T\bigl( st^{\gamma}\bigr)x\, ds,\quad t>0,\ x\in X
\end{align*}
and following E. Bazhlekova \cite{bajlekova}, R.-N. Wang, D.-H. Chen, T.-J. Xiao \cite{fractionalsectorial}, 
$$
S_{\gamma}(t):=T_{\gamma,0}(t),\mbox{   }P_{\gamma}(t):=\gamma T_{\gamma,1}(t)/t^{\gamma},\quad t>0.
$$
Recall that $(S_{\gamma}(t))_{t>0}$ is a subordinated $(g_{\gamma},I)$-regularized resolvent family generated by ${\mathcal A},$ which is not necessarily strongly continuous at zero.
In \cite{FKP}, we have proved that there exists a finite constant $M_{1}>0$ such that
\begin{align}\label{debil}
\bigl\| S_{\gamma}(t) \bigr\|+\bigl\| P_{\gamma}(t) \bigr\|\leq M_{1}t^{\gamma (\beta-1)},\quad t>0.
\end{align}
Furthermore, in \cite{nova-mono}, we have proved that there exists a finite constant $M_{2}>0$ such that
\begin{align}\label{debil-prim}
\bigl\| S_{\gamma}(t) \bigr\|\leq M_{2}t^{-\gamma},\ t\geq 1 \ \ \mbox{  and  }\ \ \bigl\| P_{\gamma}(t) \bigr\|\leq M_{2}t^{-2\gamma },\ t\geq 1.
\end{align}
Set $R_{\gamma}(t):= t^{\gamma -1}P_{\gamma}(t),$ $t>0.$

\subsection{Lebesgue spaces with variable exponents
$L^{p(x)}$}\label{karambita}

Assume $\emptyset \neq \Omega \subseteq {\mathbb R}.$
By $M(\Omega  : X)$ we denote the collection of all measurable functions $f: \Omega \rightarrow X;$ $M(\Omega):=M(\Omega : {\mathbb R}).$ Furthermore, ${\mathcal P}(\Omega)$ denotes the vector space of all Lebesgue measurable functions $p : \Omega \rightarrow [1,\infty].$
For any $p\in {\mathcal P}(\Omega)$ and $f\in M(\Omega : X),$ set
$$
\varphi_{p(x)}(t):=\left\{
\begin{array}{l}
t^{p(x)},\ t\geq 0,\ 1\leq p(x)<\infty,\\
0,\ 0\leq t\leq 1,\ p(x)=\infty,\\
\infty,\ t>1,\ p(x)=\infty 
\end{array}
\right.
$$
and
$$
\rho(f):=\int_{\Omega}\varphi_{p(x)}(\|f(x)\|)\, dx .
$$
We define Lebesgue space 
$L^{p(x)}(\Omega : X)$ with variable exponent
as follows
$$
L^{p(x)}(\Omega : X):=\Bigl\{f\in M(\Omega : X): \lim_{\lambda \rightarrow 0+}\rho(\lambda f)=0\Bigr\}.
$$
Then 
\begin{align*}
L^{p(x)}(\Omega : X)=\Bigl\{f\in M(\Omega : X):  \mbox{ there exists }\lambda>0\mbox{ such that }\rho(\lambda f)<\infty\Bigr\};
\end{align*}
see \cite[p. 73]{variable}.
For every $u\in L^{p(x)}(\Omega : X),$ we introduce the Luxemburg norm of $u(\cdot)$ in the following manner:
$$
\|u\|_{p(x)}:=\|u\|_{L^{p(x)}(\Omega :X)}:=\inf\Bigl\{ \lambda>0 : \rho(f/\lambda)    \leq 1\Bigr\}.
$$
Equipped with the above norm, the space $
L^{p(x)}(\Omega : X)$ becomes a Banach one (see e.g. \cite[Theorem 3.2.7]{variable} for scalar-valued case), coinciding with the usual Lebesgue space $L^{p}(\Omega : X)$ in the case that $p(x)=p\geq 1$ is a constant function.
For any $p\in M(\Omega),$ we set 
$$
p^{-}:=\text{essinf}_{x\in \Omega}p(x) \ \ \mbox{ and } \ \ p^{+}:=\text{esssup}_{x\in \Omega}p(x).
$$
Define
$$
C_{+}(\Omega ):=\bigl\{ p\in M(\Omega): 1<p^{-}\leq p(x) \leq p^{+} <\infty \mbox{ for a.e. }x\in \Omega \bigr \}
$$
and
$$
D_{+}(\Omega ):=\bigl\{ p\in M(\Omega): 1 \leq p^{-}\leq p(x) \leq p^{+} <\infty \mbox{ for a.e. }x\in \Omega \bigr \}.
$$
For $p\in D_{+}([0,1]),$ the space $
L^{p(x)}(\Omega : X)$ behaves nicely, with almost all fundamental properties of space $
L^{p}(\Omega : X)$ being retained; in this case, we know that 
$$
L^{p(x)}(\Omega : X)=\Bigl\{f\in M(\Omega : X):  \mbox{ for all }\lambda>0\mbox{ we have }\rho(\lambda f)<\infty\Bigr\}.
$$
Furthermore, if $p\in 
C_{+}(\Omega ),$ then $
L^{p(x)}(\Omega : X)$ is uniformly convex and thus reflexive (\cite{fan-zhao}). 

We will use the following lemma (see e.g. \cite[Lemma 3.2.20, (3.2.22); Corollary 3.3.4; p. 77]{variable} for scalar-valued case):

\begin{lem}\label{aux}
\begin{itemize}
\item[(i)] Let $p,\ q,\ r \in {\mathcal P}(\Omega),$ and let
$$
\frac{1}{q(x)}=\frac{1}{p(x)}+\frac{1}{r(x)},\quad x\in \Omega .
$$
Then, for every $u\in L^{p(x)}(\Omega : X)$ and $v\in L^{r(x)}(\Omega),$ we have $uv\in L^{q(x)}(\Omega : X)$
and
\begin{align*}
\|uv\|_{q(x)}\leq 2 \|u\|_{p(x)}\|v\|_{r(x)}.
\end{align*}
\item[(ii)] Let $\Omega $ be of a finite Lebesgue's measure, let $p,\ q \in {\mathcal P}(\Omega),$ and let $q\leq p$ a.e. on $\Omega.$ Then
 $L^{p(x)}(\Omega : X)$ is continuously embedded in $L^{q(x)}(\Omega : X).$
\item[(iii)] Let $f\in L^{p(x)}(\Omega : X),$ $g\in M(\Omega : X)$ and $0\leq \|g\| \leq \|f\|$ a.e. on $\Omega .$ Then $g\in L^{p(x)}(\Omega : X)$ and $\|g\|_{p(x)}\leq \|f\|_{p(x)}.$
\end{itemize}
\end{lem}

For more details about Lebesgue spaces with variable exponents
$L^{p(x)},$ the reader may consult \cite{m-zitane}, \cite{m-zitane-prim}, \cite{variable}, \cite{fan-zhao} and \cite{doktor}.

\section{Stepanov generalizations of almost periodic and asymptotically almost periodic functions}\label{section2}

The concept of almost periodicity was introduced by Danish mathematician H. Bohr around 1924-1926 and later generalized by many other authors (cf.  \cite{diagana},  \cite{gaston}, \cite{nova-mono} and references cited therein for more details on the subject).
Let $I={\mathbb R}$ or $I=[0,\infty),$ and let $f : I \rightarrow X$ be continuous. Given $\epsilon>0,$ we call $\tau>0$ an $\epsilon$-period for $f(\cdot)$ iff\index{$\epsilon$-period}
$
\| f(t+\tau)-f(t) \| \leq \epsilon,$ $ t\in I.
$
The set consisting of all $\epsilon$-periods for $f(\cdot)$ is denoted by $\vartheta(f,\epsilon).$ We say that $f(\cdot)$ is almost periodic, a.p. for short, iff for each $\epsilon>0$ the set $\vartheta(f,\epsilon)$ is relatively dense in $I,$ which means that
there exists $l>0$ such that any subinterval of $I$ of length $l$ meets $\vartheta(f,\epsilon)$. The vector space consisting of all almost periodic functions will be denoted by $AP(I :X).$ A function $f \in C_{b}([0,\infty) : X)$ is said to be asymptotically almost periodic  (M. Fr\' echet, 1941) iff
for every $\epsilon >0$ we can find numbers $ l > 0$ and $M >0$ such that every subinterval of $[0,\infty)$ of
length $l$ contains, at least, one number $\tau$ such that $\|f(t+\tau)-f(t)\| \leq \epsilon$ for all $t \geq M.$
It is well known that a function $f \in C_{b}([0,\infty):X)$ is asymptotically almost periodic iff
there exist functions $g \in AP([0,\infty) :X)$ and $\phi \in  C_{0}([0,\infty): X)$
such that $f = g+\phi.$
The vector space consisting of all asymptotically almost periodic functions will be denoted by $AAP([0,\infty) :X).$

Let $1\leq p <\infty,$ let $l>0,$ and let $f,\ g\in L^{p}_{loc}(I :X),$ where $I={\mathbb R}$ or $I=[0,\infty).$ We define the Stepanov `metric' by\index{Stepanov metric}
\begin{align*}
D_{S_{l}}^{p}\bigl[f(\cdot),g(\cdot)\bigr]:= \sup_{x\in I}\Biggl[ \frac{1}{l}\int_{x}^{x+l}\bigl \| f(t) -g(t)\bigr\|^{p}\, dt\Biggr]^{1/p}.
\end{align*}
Then it is well known that, for every two numbers $l_{1},\ l_{2}>0,$ there exist two positive real constants $k_{1},\ k_{2}>0$ independent of $f,\ g,$ such that
\begin{align*}
k_{1}D_{S_{l_{1}}}^{p}\bigl[f(\cdot),g(\cdot)\bigr]\leq D_{S_{l_{2}}}^{p}\bigl[f(\cdot),g(\cdot)\bigr]\leq k_{2}D_{S_{l_{1}}}^{p}\bigl[f(\cdot),g(\cdot)\bigr].
\end{align*}
The Stepanov norm of $f(\cdot)$ is introduced by\index{Stepanov norm} \index{Stepanov distance} \index{Weyl!distance} \index{Weyl!norm}
$
\| f  \|_{S_{l}^{p}}:= D_{S_{l}}^{p}[f(\cdot),0].
$

In the sequel, we assume that $l_{1}=l_{2}=1.$ It is said that a function $f\in L^{p}_{loc}(I :X)$ is Stepanov $p$-bounded, $S^{p}$-bounded shortly, iff\index{function!Stepanov bounded}
$$
\|f\|_{S^{p}}:=\sup_{t\in I}\Biggl( \int^{t+1}_{t}\|f(s)\|^{p}\, ds\Biggr)^{1/p}=\sup_{t\in I}\Biggl( \int^{1}_{0}\|f(s+t)\|^{p}\, ds\Biggr)^{1/p}<\infty.
$$
Furnished with the above norm, the space $L_{S}^{p}(I:X)$ consisted of all $S^{p}$-bounded functions is one of Banach's.
A function $f\in L_{S}^{p}(I:X)$ is said to beStepanov $p$-almost periodic, $S^{p}$-almost periodic shortly, iff the function
$
\hat{f} : I \rightarrow L^{p}([0,1] :X),
$ defined by
$$
\hat{f}(t)(s):=f(t+s),\quad t\in I,\ s\in [0,1]
$$
is almost periodic. Similarly, a function $f\in L_{S}^{p}([0,\infty):X)$ is said to be asymptotically Stepanov $p$-almost periodic, asymptotically $S^{p}$-almost periodic shortly, iff the function
$
\hat{f}(\cdot)
$ is asymptotically almost periodic.
It is well known that the space of Stepanov almost periodic functions, resp. asymptotically Stepanov almost periodic functions, denoted by $APS^{p}(I:X),$ resp. $AAPS^{p}([0,\infty):X),$ is a closed linear subspace of $L_{S}^{p}(I:X),$ resp. $L_{S}^{p}([0,\infty):X),$ and a Banach space therefore. The symbol $S^{p}_{0}([0,\infty):X)$
stands for the vector space consisting of all functions
$q\in L_{loc}^{p}([0,\infty ) : X)$ such that $\hat{q}\in C_{0}([0,\infty) : L^{p}([0,1]:X)).$

If $1\leq p<q<\infty$ and $f(\cdot)$ is (asymptotically) Stepanov $q$-almost periodic, then $f(\cdot)$ is (asymptotically) Stepanov $p$-almost periodic. Therefore, the (asymptotical) Stepanov $p$-almost periodicity of $f(\cdot)$ for some $p\in [1,\infty)$ implies the (asymptotical) Stepanov $p$-almost  periodicity of $f(\cdot).$
It is a well-known fact that if $f(\cdot)$ is an almost periodic (a.a.p.) function
then $f(\cdot)$ is also $S^p$-almost periodic (asymptotically $S^p$-a.p.) for $1\leq p <\infty.$ The converse statement is false, in general.

\section{Generalized almost periodic and generalized asymptotically almost periodic functions in  Lebesgue spaces with variable exponents
$L^{p(x)}$}\label{section3}

The following notion of Stepanov $p(x)$-boundedness is essentially different from that one introduced by T. Diagana and M. Zitane in \cite[Definition 3.10]{m-zitane} and \cite[Definition 4.5]{m-zitane-prim}, where the authors have used the condition $p\in C_{+}({\mathbb R}):$

\begin{defn}\label{daniel-toka}
Let $p\in {\mathcal P}([0,1]),$ and let $I={\mathbb R}$ or $I=[0,\infty).$ Then it is said that a function $f\in M(I : X)$ is Stepanov $p(x)$-bounded, $S^{p(x)}$-bounded in short, iff $f(\cdot +t) \in L^{p(x)}([0,1]: X)$ for all $t\in I,$ and $\sup_{t\in I} \|f(\cdot +t)\|_{p(x)}<\infty $ i.e.,
$$
\|f\|_{S^{p(x)}}:=\sup_{t\in I}\inf\Biggl\{ \lambda>0 : \int_{0}^{1}\varphi_{p(x)}\Biggl( \frac{\|f(x +t)\|}{\lambda}\Biggr)\, dx \leq 1\Biggr\}<\infty.
$$
By $L_{S}^{p(x)}(I:X)$ we denote the vector space consisting of all such functions.
\end{defn}

Immediately from definition, it follows that the space $L_{S}^{p(x)}(I:X)$ is translation invariant in the sense that, for every $f\in L_{S}^{p(x)}(I:X)$ and $\tau \in I,$ we have $f(\cdot+\tau) \in L_{S}^{p(x)}(I:X).$ This is not the case with the notion introduced in \cite{m-zitane}-\cite{m-zitane-prim}, since there the space $L_{S}^{p(x)}(I:X)$ may or may not be translation invariant depending on $p(x);$ furthermore, we would like to note that the notion introduced in these papers is meaningful even in the case that $p\in {\mathcal P}({\mathbb R}).$

We introduce the concept of (asymptotical) $S^{p(x)}$-almost periodicity as follows:

\begin{defn}\label{sasasa}
\begin{itemize}
\item[(i)]
Let $p\in {\mathcal P}([0,1]),$ and let $I={\mathbb R}$ or $I=[0,\infty).$ Then it is said that a function $f\in L_{S}^{p(x)}(I:X)$ is Stepanov $p(x)$-almost periodic, Stepanov $p(x)$-a.p. in short, iff the function $\hat{f} : I \rightarrow L^{p(x)}([0,1]: X)$ is almost periodic. By $APS^{p(x)}(I : X)$ we denote the vector space consisting of all such functions.
\item[(ii)] Let $p\in {\mathcal P}([0,1]),$ and let $I=[0,\infty).$ Then it is said that a function $f\in L_{S}^{p(x)}(I:X)$ is asymptotically Stepanov $p(x)$-almost periodic,  asymptotically Stepanov $p(x)$-a.p. in short, iff the function $\hat{f} : I \rightarrow L^{p(x)}([0,1]: X)$ is  asymptotically almost periodic. By $AAPS^{p(x)}(I : X)$ we denote the vector space consisting of all such functions; the abbreviation $S^{p(x)}_{0}([0,\infty):X)$
will be used to denote the set of all functions
$q\in  L_{S}^{p(x)}([0,\infty): X)$ such that $\hat{q}\in C_{0}([0,\infty) : L^{p(x)}([0,1]:X)).$
\end{itemize}
\end{defn}

As in the case of Stepanov $p(x)$-boundedness, the space $APS^{p(x)}(I : X)$ is translation invariant in the sense that, for every $f\in APS^{p(x)}(I : X)$ and $\tau \in I,$ we have $f(\cdot+\tau) \in APS^{p(x)}(I : X).$ A similar statement holds for the space $AAPS^{p(x)}([0,\infty) : X).$

It can be simply checked that the notions of (asymptotical) Stepanov $p(x)$-boundedness and (asymptotical) Stepanov $p(x)$-almost periodicity are equivalent with those ones introduced in the previous section, provided that $p(x)\equiv p\geq 1$ is a constant function.

Furnished with the norm $\|\cdot\|_{S^{p(x)}}$, the space $L_{S}^{p(x)}(I:X)$ consisted of all $S^{p}$-bounded functions is a Banach one and this space is continuously embedded in $L_{S}^{1}(I:X),$ for any $p\in {\mathcal P}([0,1]).$ Furthermore, it can be simply proved that 
$APS^{p(x)}(I : X)$ ($AAPS^{p(x)}(I : X)$ with $I=[0,\infty)$) is a closed subspace of $L_{S}^{p(x)}(I:X)$ and therefore Banach space itself, for any $p\in {\mathcal P}([0,1])$.

If $p\in {\mathcal P}([0,1]),$ then Lemma \ref{aux}(ii) implies that $L^{p(x)}([0,1] : X) \hookrightarrow L^{1}([0,1] : X),$ where $\hookrightarrow$
denotes a continuous embedding, so that $L^{p(x)}_{S}(I : X)\hookrightarrow L^{1}_{S}(I : X),$ as well. Therefore,
the following holds:

\begin{prop}\label{potapanje}
Assume that $p\in {\mathcal P}([0,1])$. Then we have the following:
\begin{itemize}
\item[(i)] $L^{p(x)}_{S}(I : X)\hookrightarrow L^{1}_{S}(I : X).$ 
\item[(ii)] $APS^{p(x)}(I : X) \hookrightarrow APS^{1}(I : X)$ and  $AAPS^{p(x)}([0,\infty) : X) \hookrightarrow AAPS^{1}([0,\infty) : X).$
\end{itemize}
\end{prop}

We can similarly prove the following proposition: 

\begin{prop}\label{d+}
Assume that $p\in D_{+}([0,1])$ and $1 \leq p^{-}\leq p(x) \leq p^{+} <\infty$ for a.e. $x\in [0,1] $. Then we have the following: 
\begin{itemize}
\item[(i)] $L^{p^{+}}_{S}(I : X)\hookrightarrow  L^{p(x)}_{S}(I : X) \hookrightarrow L^{p^{-}}_{S}(I : X).$ 
\item[(ii)] $APS^{p^{+}}(I : X)\hookrightarrow  APS^{p(x)}(I : X) \hookrightarrow APS^{p^{-}}(I : X)$ and $AAPS^{p^{+}}([0,\infty) : X)\hookrightarrow  AAPS^{p(x)}([0,\infty) : X) \hookrightarrow AAPS^{p^{-}}([0,\infty) : X).$
\end{itemize}
\end{prop}

Now we will prove that any almost periodic function is $S^{p(x)}$-almost periodic, for any $p\in {\mathcal P}([0,1]):$

\begin{prop}\label{propa}
Let $p\in {\mathcal P}([0,1]),$ and let $f : I \rightarrow X$ be almost periodic. Then $f(\cdot)$ is $S^{p(x)}$-almost periodic.
\end{prop}

\begin{proof}
To prove that $f(\cdot)$ is $S^{p(x)}$-bounded and $\|f\|_{L_{S}^{p(x)}}\leq \|f\|_{\infty},$ it suffices to show that, for every $t\in {\mathbb R},$ we have:
\begin{align}\label{spbounded}
\bigl[ \|f\|_{\infty},\infty \bigr) \subseteq \Biggl\{ \lambda>0 : \int_{0}^{1}\varphi_{p(x)}\Biggl( \frac{\|f(x +t)\|}{\lambda}\Biggr)\, dx \leq 1\Biggr\}.
\end{align}
For $\lambda \geq \|f\|_{\infty}, $ we have $\|f(x +t)\|/\lambda \leq 1,$ $t\in I.$ It can be simply perceived that, in this case,
$$
\varphi_{p(x)}\Biggl( \frac{\|f(x +t)\|}{\lambda}\Biggr)\leq 1,\quad t\in I,
$$
so that the integrand does not exceed $1;$
as a matter of fact,
by definition of $\varphi_{p(x)}(\cdot)$, we only need to observe that, for every $x\in [0,1]$ with $p(x)<\infty ,$ we have $(\| f(t+x)\|/\lambda)^{p(x)}\leq 1^{p(x)}=1,$ $t\in I.$ Hence, \eqref{spbounded} holds.
Using the uniform continuity of $f(\cdot)$ and a similar argumentation, we can show that the function $\hat{f} :I \rightarrow L^{p(x)}([0,1]:X)$ is uniform continuous. 
For direct proof of almost periodicity of function $\hat{f} :I \rightarrow L^{p(x)}([0,1]:X),$ we can argue as follows. For $\epsilon>0$ given as above,
there is a finite number $l>0$
such that any subinterval $I'$ of $I$ of length $l$ contains a number
$\tau \in I'$ such that
$\| f(t+\tau)-f(t) \|\leq \epsilon, $ $ t\in I.
$ It suffices to observe that, for this $\epsilon>0,$ we can choose the same length $l>0$ and the same $\epsilon$-almost period $\tau$ from $I'$ ensuring the validity of inequality $\| \hat{f}(t+\tau +\cdot)-\hat{f}(t+\cdot) \|_{ L^{p(x)}([0,1]:X)}\leq \epsilon, $ $ t\in I:
$ in order to see that the last inequality holds true, we only need to prove that, for every $t\in I,$ we have
$$
[\epsilon,\infty) \subseteq \Biggl\{\lambda>0 :  \int^{1}_{0}\varphi_{p(x)}\Biggl(  \frac{\| f(t+\tau+x)-f(t+x)\|}{\lambda} \Biggr)\, dx \leq 1\Biggl\}.
$$ 
Indeed, if $\lambda \geq \epsilon,$ then $\| f(t+\tau+x)-f(t+x)\|/\lambda\leq 1,$ $t\in I$ and the integrand cannot exceed $1:$ this simply follows from definition of $\varphi_{p(x)}(\cdot)$ and observation that, for every $x\in [0,1]$ with $p(x)<\infty ,$ we have $(\| f(t+\tau+x)-f(t+x)\|/\lambda)^{p(x)}\leq 1^{p(x)}=1,$ $t\in I.$ The proof of the proposition is thereby complete.
\end{proof}

We can similarly prove the following proposition:

\begin{prop}\label{propa-asym}
Let $p\in {\mathcal P}([0,1]),$ and let $f : [0,\infty) \rightarrow X$ be asymptotically almost periodic. Then $f(\cdot)$ is asymptotically $S^{p(x)}$-almost periodic.
\end{prop}

Taking into account Proposition \ref{potapanje}(ii) and the method employed in the proof of Proposition \ref{propa}, we can state the following: 

\begin{prop}\label{potapanj-prime}
Assume that $p\in {\mathcal P}([0,1])$ and $f\in L_{S}^{p(x)}(I:X)$. Then the following holds:
\begin{itemize}
\item[(i)] $L^{\infty}(I: X) \hookrightarrow L^{p(x)}_{S}(I : X)\hookrightarrow L^{1}_{S}(I : X).$
\item[(ii)] 
$
AP(I:X) \hookrightarrow APS^{p(x)}(I : X) \hookrightarrow APS^{1}(I : X)
$ and $
AAP( [0,\infty):X) \hookrightarrow AAPS^{p(x)}( [0,\infty) : X) \hookrightarrow AAPS^{1}( [0,\infty) : X).
$
\item[(iii)] The continuity (uniform continuity) of $f(\cdot)$ implies continuity (uniform continuity) of $\hat{f}(\cdot).$
\end{itemize}
\end{prop}

In general case, we have the following:

\begin{prop}\label{ghf=prop}
Assume that $p,\ q\in {\mathcal P}([0,1])$ and $p\leq q$ a.e. on $[0,1].$ Then we have:
\begin{itemize}
\item[(i)] $L^{q(x)}_{S}(I : X)\hookrightarrow L^{p(x)}_{S}(I : X).$
\item[(ii)] 
$
APS^{q(x)}(I : X) \hookrightarrow APS^{p(x)}(I : X)
$ and $
AAPS^{q(x)}([0,\infty) : X) \hookrightarrow AAPS^{p(x)}([0,\infty) : X).
$
\item[(iii)] If $p\in D_{+}([0,1]),$ then 
\begin{align*}
L^{\infty}(I:X) \cap APS^{p(x)}(I : X)=L^{\infty}(I:X) \cap APS^{1}(I:X)
\end{align*}
and
\begin{align*}
L^{\infty}([0,\infty):X) \cap AAPS^{p(x)}([0,\infty) : X)=L^{\infty}([0,\infty):X) \cap AAPS^{1}([0,\infty):X).
\end{align*}
\end{itemize}
\end{prop}

\begin{proof}
We will prove only (iii) for almost periodicity. Keeping in mind Proposition \ref{d+}(ii), it suffices to assume that $p(x)\equiv p>1.$ Then, clearly, $L^{\infty}(I:X) \cap APS^{p}(I : X) \subseteq L^{\infty}(I:X) \cap APS^{1}(I:X)$ and it remains to be proved the opposite inclusion. So, let $f\in L^{\infty}(I:X) \cap APS^{1}(I:X).$ The required conclusion is a consequence of elementary definitions and following simple calculation, which is valid for any $t,\ \tau \in {\mathbb R}:$
\begin{align*}
\Biggl[ \int^{t+1}_{t}\bigl\|f(\tau +s) & -f(s)\bigr\|^{p}\, ds \Biggr]^{1/p}
\\ & \leq \Biggl[ \int^{t+1}_{t} \bigl( 2\|f\|_{\infty}\bigr)^{p-1}\bigl\|f(\tau +s)  -f(s)\bigr\|\, ds \Biggr]^{1/p}
\\ & =\bigl( 2\|f\|_{\infty}\bigr)^{(p-1)/p}\Biggl[ \int^{t+1}_{t}\bigl\|f(\tau +s)  -f(s)\bigr\|\, ds \Biggr]^{1/p}.
\end{align*}
\end{proof}

\begin{rem}\label{folner}
It is well known that $APS^{p(x)}(I : X)$ can be strictly contained in $APS^{1}(I:X),$ even in the case that $p(x)\equiv p>1$ is a constant function. For example, H. Bohr and E. F$\o$lner have proved that, for any given number $p>1$, we can construct a Stepanov almost periodic function defined on the whole real axis that is not Stepanov $p$-almost periodic (see \cite[Example, p. 70]{bohr-folner}). The same example shows that $AAPS^{p}([0,\infty) : X)$ can be strictly contained in $AAPS^{1}([0,\infty):X)$ for $p>1$ (see e.g. \cite[Lemma 1]{hernan1}).
\end{rem}

\begin{rem}\label{propa-lis}
Proposition \ref{propa} and Proposition \ref{propa-asym} can be simply deduced by using Proposition \ref{ghf=prop}(ii) and the equalities $AP(I : X)=APS^{\infty}(I:X) \cap C(I : X),$ $AAP([0,\infty) : X)=AAPS^{\infty}([0,\infty):X) \cap C([0,\infty) : X)$, which can be proved almost trivially. 
\end{rem}

Now we would like to present the following illustrative example:

\begin{example}\label{ogran-prop}
Define sign$(0):=0.$ Then, for every almost periodic function $f: {\mathbb R} \rightarrow {\mathbb R},$ we have that the function $F(\cdot):=$sign$(f(\cdot))$ is Stepanov $1$-almost periodic (\cite{188}). Since $F\in L^{\infty}({\mathbb R}),$ Proposition \ref{ghf=prop}(iii) yields that the function $F(\cdot)$ is Stepanov $p$-almost periodic for any $p\geq 1,$ while Proposition \ref{potapanj-prime}(i) yields that the function $F(\cdot)$ is Stepanov $p(x)$-bounded for any $p\in {\mathcal P}([0,1]).$ Due to Proposition \ref{d+}(ii), we have 
$F\in APS^{p(x)}({\mathbb R} : {\mathbb C})$ for any $p\in D_{+}([0,1]).$

Consider now the case that $f(x):=\sin x+\sin \sqrt{2}x,$ $x\in {\mathbb R}$ and $p(x):=1-\ln x,$ $x\in [0,1].$ We will prove that $F\notin APS^{p(x)}({\mathbb R} : {\mathbb C}).$ Speaking-matter-of-factly, it is sufficient to show that, for every $\lambda \in (0,2/e)$ and for every $l>0,$ we can find an interval $I\subseteq {\mathbb R}$ of length $l>0$ such that, for every $\tau \in I,$ there exists $t\in {\mathbb R}$ such that 
\begin{align}
\notag \int^{1}_{0} \Bigl( \frac{1}{\lambda}\Bigr)^{1-\ln x}&\Bigl| \text{sign}\bigl[\sin  (x+t+\tau) +\sin \sqrt{2} (x+t+\tau)\bigr]
\\ & \label{singular}-\text{sign}\bigl[\sin (x+t) +\sin \sqrt{2}(x+t)\bigr] \Bigr|^{1-\ln x}\, dx =\infty.
\end{align}
Let  $\lambda \in (0,2/e)$ and $l>0$ be given. Take arbitrarily any interval $I \subseteq {\mathbb R} \setminus \{0\}$ of length $l$ and after that take arbitrarily any number $\tau \in I.$ Since $(1/\lambda)^{1-\ln x}\geq 1/x,$ $x\in [0,1]$ and $1-\ln x\geq 1,$ $x\in [0,1],$ a continuity argument shows that it is enough to prove the existence of a number $t\in {\mathbb R}$ such that
\begin{align}\label{cvb-propa}
\bigl[\sin  (t+\tau) +\sin \sqrt{2} (t+\tau)\bigr] \cdot \bigl[\sin t +\sin \sqrt{2}t\bigr] <0.
\end{align}
If $\sin \tau +\sin \sqrt{2} \tau>0$ ($\sin \tau +\sin \sqrt{2} \tau <0$), then we can take $t\sim 0-$ ($t\sim 0+$). Hence, we assume henceforward $\sin \tau +\sin \sqrt{2} \tau=0$ and $\tau \neq 0.$ There exist two possibilities:
$$
\tau \in \frac{2{\mathbb Z}\pi}{1+\sqrt{2}} \setminus \{0\}\ \ \mbox{ or }\ \ \tau \in \frac{(2{\mathbb Z}+1)\pi}{\sqrt{2}-1}.
$$
In the first case, take $t_{0}=\frac{\pi}{\sqrt{2}-1}.$ Then an elementary argumentation shows that $\tau +t_{0}\notin \frac{2{\mathbb Z}\pi}{1+\sqrt{2}} \cup \frac{(2{\mathbb Z}+1)\pi}{\sqrt{2}-1}$ so that $\sin  (t_{0}+\tau) +\sin \sqrt{2} (t_{0}+\tau)\neq 0.$ If $\sin  (t_{0}+\tau) +\sin \sqrt{2} (t_{0}+\tau)>0$ ($\sin  (t_{0}+\tau) +\sin \sqrt{2} (t_{0}+\tau)<0$), then for $t$ satisfying \eqref{cvb-propa} we can take any number belonging to a small left/right interval around $t_{0}$ for which $\sin t +\sin \sqrt{2}t<0$ ($\sin t +\sin \sqrt{2}t>0$). In the second case, there exists an integer $m\in {\mathbb Z}$ such that $\tau=\frac{(2m+1)\pi}{\sqrt{2}-1}$ and we can take $t_{0}=\frac{(-2m+1)\pi}{\sqrt{2}-1}.$ Then $\tau+t_{0}=\frac{2\pi}{\sqrt{2}-1}$ and $\sin  (t_{0}+\tau) +\sin \sqrt{2} (t_{0}+\tau)\neq 0,$ so that we can use a trick similar to that used in the first case. Let us only mention in passing that, with the notion introduced in \cite{toka-marek-prim}, the function $F(\cdot)$ cannot be $S^{p(x)}$-almost automorphic, as well.

The situation is quite different if we consider the case that $f(x):=\sin x,$ $x\in {\mathbb R}.$ Then $F(\cdot)$ is Stepanov $p(x)$-almost periodic for any $p\in {\mathcal P}([0,1]).$ Speaking-matter-of-factly, it can be easily shown that the mapping $\hat{F} : {\mathbb R}\rightarrow L^{p(x)}[0,1] $ is continuous and $\| F(t+\tau +\cdot)-F(t+\cdot)\|_{ L^{p(x)}[0,1] }=0$ for all $t\in {\mathbb R}$ and $\tau \in 2\pi {\mathbb Z}.$ This, in turn, implies the claimed statement.
\end{example}

Keeping in mind the proofs of Proposition \ref{propa}, \cite[Proposition 3.5]{toka-marek-prim} and \cite[Lemma 1]{hernan1}, we can clarify the following result:

\begin{prop}\label{tricky}
Suppose that $p\in {\mathcal P}([0,1])$ and $ f: [0,\infty)\rightarrow X$ is an asymptotically $S^{p(x)}$-almost periodic
function. Then there are uniquely determined $S^{p(x)}$-bounded functions $g: {\mathbb R} \rightarrow X$ and
$q: [0,\infty)\rightarrow X$ satisfying the following conditions:
\begin{itemize}
\item[(i)] $g$ is $S^{p(x)}$-almost periodic,
\item[(ii)] $\hat{q}$ belongs to the class $C_{0}([0,\infty) : L^{p(x)}([0,1]:X)),$
\item[(iii)] $f(t)=g(t)+q(t)$ for all $t\geq 0.$
\end{itemize}
Moreover, there exists an increasing sequence $(t_{n})_{n\in {\mathbb N}}$ of positive reals such that $\lim_{n\rightarrow \infty}t_{n}=\infty$
and $g(t)=\lim_{n\rightarrow \infty}f(t+ t_{n})$ a.e. $t\geq 0.$
\end{prop}

\begin{rem}\label{weyl-podx}
The definition of an (equi-)Weyl $p(x)$-almost periodic function (see e.g. \cite{nova-mono} for the case that $p(x)\equiv p\in [1,\infty)$) can be introduced as follows: Suppose $I={\mathbb R}$ or $I=[0,\infty).$  
Let $p\in {\mathcal P}(I)$ and $f(\cdot +\tau)\in L^{p(x)}( K : X)$ for any $\tau \in I$ and any compact subset $K$ of $I.$ 
\begin{itemize}
\item[(i)] It is said that the function $f(\cdot)$ is equi-Weyl-$p(x)$-almost periodic, $f\in e-W_{ap}^{p(x)}(I:X)$ for short, iff for each $\epsilon>0$ we can find two real numbers $l>0$ and $L>0$ such that any interval $I'\subseteq I$ of length $L$ contains a point $\tau \in  I'$ such that
\begin{align*}
\sup_{t\in I}\Biggl[ l^{(-1)/p(t)}\bigl \| f(\cdot+\tau) -f(\cdot)\bigr\|_{L^{p(x)}[t,t+l]}\Biggr] \leq \epsilon .
\end{align*}
\item[(ii)] It is said that the function $f(\cdot)$ is Weyl-$p(x)$-almost periodic, $f\in W_{ap}^{p(x)}(I: X)$ for short, iff for each $\epsilon>0$ we can find a real number $L>0$ such that any interval $I'\subseteq I$ of length $L$ contains a point $\tau \in  I'$ such that
\begin{align*}
\lim_{l\rightarrow \infty} \sup_{t\in I}\Biggl[ l^{(-1)/p(t)}\bigl \| f(\cdot+\tau) -f(\cdot)\bigr\|_{L^{p(x)}[t,t+l]}\Biggr]  \leq \epsilon .
\end{align*}
\end{itemize}
The notion of (equi-)Weyl $p(x)$-almost periodicity as well as the corresponding notion for Besicovitch classes of almost periodic functions will not attract our attention here. We will also skip all details concerning asymptotical  $p(x)$-almost periodicity for Weyl and Besicovitch classes.
\end{rem}

\section{Generalized two-parameter almost periodic type functions and composition principles}\label{dodato}

Assume that $(Y,\|\cdot \|_{Y})$ is a complex Banach space, as well as that $I={\mathbb R}$ or $I=[0,\infty).$ 
By $C_{0}([0,\infty) \times Y : X)$ we denote the space consisting of all
continuous functions $h : [0,\infty)  \times Y \rightarrow X$ such that $\lim_{t\rightarrow \infty}h(t, y) = 0$ uniformly for $y$ in any compact subset of $Y .$ A continuous function $f : I  \times Y  \rightarrow X$ is said to be uniformly continuous on
bounded sets, uniformly for $t \in I$ iff for every $	\epsilon > 0$ and every bounded subset $K$ of $Y$ there
exists a number $\delta_{\epsilon,K }> 0$ such that $\|f(t , x)- f (t, y)\| \leq \epsilon$ for all $ t \in I$ and all $x,\ y \in K$ satisfying that $\|x-y\|\leq \delta_{\epsilon,K }.$ If $f : I  \times Y  \rightarrow X,$ set $\hat{f}(t , y):=f(t +\cdot, y),$ $t\geq 0,$ $y\in Y.$

We need to recall the following well-known definition (see e.g. \cite{nova-mono} for more details):

\begin{defn}\label{definicija}
Let $1\leq p <\infty.$
\begin{itemize}
\item[(i)]\index{function!two-parameter!almost periodic}
A function $f : I \times Y \rightarrow X$ is said to be almost periodic iff $f (\cdot, \cdot)$ is bounded, continuous as well as for every $\epsilon>0$ and every compact
$K\subseteq Y$ there exists $l(\epsilon,K) > 0$ such that every subinterval $J\subseteq I$ of length $l(\epsilon,K)$ contains a number $\tau$ with the property that $\|f (t +\tau , y)- f (t, y)\| \leq \epsilon$ for all $t \in  I,$ $ y \in K.$ The collection of such functions will be denoted by $AP(I \times Y : X).$
\item[(ii)] A function $f : [0,\infty)  \times Y \rightarrow X$ is said to be asymptotically almost periodic iff it is bounded continuous and admits a\index{function!two-parameter!asymptotically almost periodic}
decomposition $f = g + q,$ where $g \in AP([0,\infty)  \times Y : X)$ and $q\in C_{0}([0,\infty)  \times Y : X).$ Denote by
 $AAP([0,\infty)  \times Y : X) $ the vector space consisting of all such functions.
\end{itemize}
\end{defn}

The notion of (asymptotical) Stepanov $p(x)$-almost periodicity for the functions depending on two parameters is introduced as follows:

\begin{defn}\label{stepa-px}
Let $p\in {\mathcal P}([0,1]).$
\begin{itemize}
\item[(i)] A function $f : I \times Y \rightarrow X$ is called Stepanov $p(x)$-almost periodic, $S^{p(x)}$-almost periodic for short, iff $\hat{f} : I  \times Y  \rightarrow L^{p(x)}([0,1]:X)$ is almost periodic. The vector space consisting of all such functions will be denoted by $APS^{p(x)}(I \times Y : X).$
\item[(ii)]
A function $f : [0,\infty)  \times Y \rightarrow X$
is said to be asymptotically $S^{p(x)}$-almost periodic\index{function!two-parameter!asymptotically Stepanov almost periodic}
iff $\hat{f}: [0,\infty)  \times Y \rightarrow  L^{p(x)}([0,1]:X)$ is asymptotically almost periodic. The vector space consisting of all such functions will be denoted by $AAPS^{p(x)}([0,\infty) \times Y : X).$
\end{itemize}
\end{defn}

The proof of following proposition is very similar to the proof of \cite[Lemma 2.2.6]{nova-mono} and therefore omitted.

\begin{prop}\label{tricky-prim}
Let $p\in {\mathcal P}([0,1]).$
Suppose that $ f: [0,\infty) \times Y \rightarrow X$ is an asymptotically $S^{p(x)}$-almost periodic
function. Then there are two functions $g: {\mathbb R} \times Y \rightarrow X$ and
$q: [0,\infty) \times Y \rightarrow X$ satisfying that for each $y\in Y$ the functions
$g(\cdot,y)$ and
$q(\cdot,y)$ are Stepanov $p(x)$-bounded, as well as that
the following holds:
\begin{itemize}
\item[(i)] $\hat{g} : {\mathbb R} \times Y \rightarrow L^{p(x)}([0,1]:X)$ is almost periodic,
\item[(ii)] $\hat{q} \in C_{0}([0,\infty) \times Y : L^{p(x)}([0,1]:X)),$
\item[(iii)] $f(t,y)=g(t,y)+q(t,y)$ for all $t\geq 0$ and $y\in Y.$
\end{itemize}
Moreover, for every compact set $K \subseteq Y,$ there exists an increasing sequence $(t_{n})_{n\in {\mathbb N}}$ of positive reals such that $\lim_{n\rightarrow \infty}t_{n}=\infty$
and $g(t,y)=\lim_{n\rightarrow \infty}f(t+ t_{n},y)$ for all $ y\in Y$  and a.e. $t\geq 0.$
\end{prop}

In \cite[Theorem 2.7.1, Theorem 2.7.2]{nova-mono}, we have slightly improved the important composition principle atributed to W. Long, S.-H. Ding \cite[Theorem 2.2]{comp-adv}. Further refinements for $S^{p(x)}$-almost periodicity can be deduced similarly, with appealing to Lemma \ref{aux}(i)-(iii) and the arguments employed in the proof of \cite[Theorem 2.2]{comp-adv}: 

\begin{thm}\label{vcb-show}
Let $I={\mathbb R}$ or $I=[0,\infty),$ and let $p\in {\mathcal P}([0,1]).$
Suppose that the following conditions hold:
\begin{itemize}
\item[(i)] $f \in APS^{p(x)}(I \times Y : X)  $ and there exist a function $r\in {\mathcal P}([0,1])$ such that $ r(\cdot)\geq \max (p(\cdot), p(\cdot)/p(\cdot) -1)$ and a function $ L_{f}\in L_{S}^{r(x)}(I) $ such that:
\begin{align}\label{vbnmp}
\| f(t,x)-f(t,y)\| \leq L_{f}(t)\|x-y\|_{Y},\quad t\in I,\ x,\ y\in Y;
\end{align}
\item[(ii)] $u \in APS^{p(x)} (I: Y),$ and there exists a set ${\mathrm E} \subseteq I$ with $m ({\mathrm E})= 0$ such that
$ K :=\{u(t) : t \in I \setminus {\mathrm E}\}$
is relatively compact in $Y;$ here, $m(\cdot)$ denotes the Lebesgue measure.
\end{itemize}
Define $q\in {\mathcal P}([0,1])$ by
$q(x):=p(x)r(x)/p(x)+r(x),$ if $x\in [0,1]$ and $r(x)<\infty,$ $q(x):=p(x),$ if $x\in [0,1]$ and $r(x)=\infty.$ Then $q(x)\in [1, p(x))$ for $x\in [0,1],$ $r(x)<\infty$ and $f(\cdot, u(\cdot)) \in APS^{q(x)}(I : X).$
\end{thm}

Concerning asymptotical two-parameter Stepanov $p(x)$-almost periodicity, we can deduce the following composition principles with $X=Y;$ the proof is very similar to those of \cite[Proposition 2.7.3, Proposition 2.7.4]{nova-mono} established in the case of constant functions $p,\ q,\ r:$

\begin{prop}\label{bibl}
Let $I =[0,\infty),$ and let $p\in {\mathcal P}([0,1]).$
Suppose that the following conditions hold:
\begin{itemize}
\item[(i)] $g \in APS^{p(x)}(I \times X : X)  ,$ there exist a function $r\in {\mathcal P}([0,1])$ such that $ r(\cdot)\geq \max (p(\cdot), p(\cdot)/p(\cdot) -1)$ and a function $ L_{g}\in L_{S}^{r(x)}(I) $ such that \eqref{vbnmp} holds with the function $f(\cdot, \cdot)  $ replaced by the function $g(\cdot, \cdot)  $ therein.
\item[(ii)] $v \in APS^{p(x)}(I:X),$ and there exists a set ${\mathrm E} \subseteq I$ with $m ({\mathrm E})= 0$ such that
$ K =\{v(t) : t \in I \setminus {\mathrm E}\}$
is relatively compact in X.
\item[(iii)] $f(t,x)=g(t,x)+q(t,x)$ for all $t\geq 0$ and $x\in X,$ where $\hat{q}\in C_{0}([0,\infty) \times X : L^{q(x)}([0,1]:X))$
with $q(\cdot)$ defined as above;
\item[(iv)] $u(t)=v(t)+\omega(t) $ for all $t\geq 0,$ where $\hat{\omega}\in C_{0}([0,\infty) : L^{p(x)}([0,1]:X)).$
\item[(v)]  There exists a set $E' \subseteq I$ with $m (E')= 0$ such that
$ K' =\{u(t) : t \in I \setminus E'\}$
is relatively compact in $ X.$
\end{itemize}
Then $f(\cdot, u(\cdot)) \in AAPS^{q(x)}(I : X).$
\end{prop}

\section{Generalized (asymptotical) almost periodicity in Lebesgue spaces with variable exponents
$L^{p(x)}:$ action of convolution products}\label{ne-mozem}

Throughout this section, we assume that $p\in {\mathcal P}([0,1])$ and a multivalued linear operator ${\mathcal A}$ fulfills the condition (P).
We will first investigate infinite convolution products. The results obtained can be simply incorporated in the study of existence and uniqueness of almost periodic solutions of the following abstract Cauchy differential inclusion of first order
$$
u^{\prime}(t)\in {\mathcal A}u(t)+g(t),\quad t\in {\mathbb R}
$$
and the following abstract Cauchy relaxation differential inclusion
\begin{align}\label{left-bruka}
D_{t,+}^{\gamma}u(t)\in -{\mathcal A}u(t)+g(t),\ t\in {\mathbb R},
\end{align}
where $D_{t,+}^{\gamma}$ denotes the Weyl-Liouville fractional derivative of order $\gamma \in (0,1)$ and $g: {\mathbb R} \times X \rightarrow X$ satisfies certain assumptions; see \cite{nova-mono} for further information in this direction. Keeping in mind composition principles clarified in the previous section, it is almost straightforward to reformulate some known results concerning semilinear analogues of the above inclusions (see e.g. \cite[Theorem 2.7.6-Theorem 2.7.9; Theorem 2.9.10-Theorem 2.9.11; Theorem 2.9.17-Theorem 2.9.18]{nova-mono}); because of that, this question 
will not be examined here for the sake of brevity.

We start by stating
the following generalization of \cite[Proposition 2.11]{EJDE} (the reflexion at zero keeps the spaces of Stepanov $p$-almost periodic functions unchanged, which may or may not be the case with the spaces of Stepanov $p(x)$-almost periodic functions):

\begin{prop}\label{ravi-and-variable}
Suppose that $q\in {\mathcal P}([0,1]),$ $1/p(x) +1/q(x)=1$
and $(R(t))_{t> 0}\subseteq L(X,Y)$ is a strongly continuous operator family satisfying that
$M:=\sum_{k=0}^{\infty}\|R(\cdot +k)\|_{L^{q(x)}[0,1]}<\infty .$ If $\check{g} : {\mathbb R} \rightarrow X$ is $S^{p(x)}$-almost periodic, then the function $G: {\mathbb R} \rightarrow Y,$ given by
\begin{align}\label{wer}
G(t):=\int^{t}_{-\infty}R(t-s)g(s)\, ds,\quad t\in {\mathbb R},
\end{align}
is well-defined and almost periodic.
\end{prop}

\begin{proof}
Without loss of generality, we may assume that $X=Y.$
It is clear that, for every $ t\in {\mathbb R},$ we have that $G(t)=\int^{\infty}_{0}R(s)g(t-s)\, ds$ and that the last integral is absolutely convergent due to Lemma \ref{aux}(i) and $S^{p(x)}$-boundedness of function $\check{g}(\cdot):$\index{H\"older inequality}
\begin{align*}
 \int^{\infty}_{0}\|R(s)\|& \| g(t-s)\|\, ds=\sum _{k=0}^{\infty}  \int^{k+1}_{k}
\|R(s)\|\| g(t-s)\|\, ds 
\\& =\sum _{k=0}^{\infty}  \int^{1}_{0}
\|R(s+k)\|\| g(t-s-k)\|\, ds \\& \leq 2
\sum _{k=0}^{\infty}\|R(\cdot +k)\|_{L^{q(x)}([0,1] : X)} \|g(t-k-\cdot)\|_{L^{p(x)}([0,1] : X)}
\\& \leq 2M\sup_{t\in {\mathbb R}}\|\check{g}(\cdot-t)\|_{L^{p(x)}([0,1] : X)},
\end{align*}
for any $t\in {\mathbb R}.$
Let a number $\epsilon>0$ be fixed.
Then there is a finite number $l>0$
such that any subinterval $I$ of ${\mathbb R}$ of length $l$ contains a number
$\tau \in I$ such that
$\| \check{g}(t-\tau+\cdot)-\check{g}(t+\cdot) \|_{L^{p(x)}([0,1] : X)}\leq \epsilon, $ $ t\in {\mathbb R}.
$ Invoking Lemma \ref{aux}(i) and this fact, we get 
\begin{align*}
\| G(t+\tau)&-G(t)\|
\\ & \leq \int^{\infty}_{0}\|R(r) \|\cdot \| g(t+\tau -r)-g(t-r) \| \, dr
\\ & =\sum _{k=0}^{\infty} \int^{k+1}_{k}\|R(r) \| \cdot \| g(t+\tau -r)-g(t-r) \| \, dr
\\ & =\sum _{k=0}^{\infty} \int^{1}_{0}\|R(r+k) \| \cdot \| g(t+\tau -r-k)-g(t-r-k) \| \, dr
\\ & \leq 2\sum _{k=0}^{\infty} \|R(\cdot +k)\|_{L^{q(x)}[0,1]}\| g(t+\tau -\cdot-k)-g(t-\cdot-k) \|_{L^{p(x)}[0,1]}
\\ & =2\sum _{k=0}^{\infty} \|R(\cdot +k)\|_{L^{q(x)}[0,1]}\| \check{g}(\cdot-t-\tau +k)-\check{g}(\cdot-t+k) \|_{L^{p(x)}[0,1]}
\\ & \leq 2\epsilon \sum _{k=0}^{\infty} \|R(\cdot +k)\|_{L^{q(x)}[0,1]}= 2M \epsilon,\quad t\in {\mathbb R},
\end{align*}
which clearly implies that the set of all $\epsilon$-periods of $G(\cdot)$ is relatively dense in ${\mathbb R}$. It remains to be proved the uniform continuity of $G(\cdot).$ Since $\hat{\check{g}}(\cdot)$ is uniformly continuous, we have the existence of a number $\delta \in (0,1)$ such that
\begin{align}\label{loz}
\|\check{g}(\cdot-t')-\check{g}(\cdot -t)\|_{L^{p(x)}[0,1]}<\epsilon,\mbox{  provided  }t,\ t'\in {\mathbb R} \mbox{ and } |t-t'|<\delta.
\end{align}
For any $\delta'\in (0,\delta),$ the above computation with $\tau=\delta'=t'-t$ and \eqref{loz} together imply that, for every $t\in {\mathbb R},$
\begin{align*}
\bigl\| & G(t+\delta')-G(t)\bigr\|
\\& \leq 2\sum _{k=0}^{\infty} \|R(\cdot +k)\|_{L^{q(x)}[0,1]}\| \check{g}(\cdot-t'+k)-\check{g}(\cdot-t+k) \|_{L^{p(x)}[0,1]}
\\& \leq 2\epsilon \sum _{k=0}^{\infty} \|R(\cdot +k)\|_{L^{q(x)}[0,1]}= 2M \epsilon.
\end{align*}
This completes the proof of proposition.
\end{proof}

\begin{example}\label{zaeljjjo}
\begin{itemize}
\item[(i)] Suppose that $\beta \in (0,1)$ and $(R(t))_{t>0}=(T(t))_{t>0}$ is a degenerate semigroup generated by 
${\mathcal A}.$ Let us recall that there exists a finite constant $M>0$ such that
$\|T(t)\|\leq Mt^{\beta -1},$ $t\in (0,1]$ and $\|T(t)\|\leq Me^{-ct},$ $t\geq 1.$ Let $p_{0}>1$ be such that
$$
\frac{p_{0}}{p_{0}-1}(\beta -1) \leq -1,
$$ 
let $p\in {\mathcal P}([0,1]),$ and let $\|T(\cdot)\|_{L^{q(x)}[0,1]}<\infty .$
Assume that we have constructed a function $\check{g}\in APS^{p(x)}({\mathbb R} : X)$ such that $\check{g} \notin APS^{p}({\mathbb R} : X)$ for all $p\geq p_{0}$ (Question: Could we manipulate here somehow with the construction established in \cite[Example, p. 70]{bohr-folner}?) Then, in our concrete situation, \cite[Proposition 2.11]{EJDE} cannot be applied since 
$$
\frac{p}{p-1}(\beta -1) \leq -1,\quad p\in [1,p_{0}).
$$
Now we will briefly explain that $\sum_{k=0}^{\infty}\|R(\cdot +k)\|_{L^{q(x)}[0,1]}<\infty ,$ showing that Proposition \ref{ravi-and-variable} is applicable. Strictly speaking, for $k=0,$ $\|T(\cdot)\|_{L^{q(x)}[0,1]}<\infty$ by our assumption, while, for $k\geq 1,$ it can be simply shown that $\|R(\cdot +k)\|_{L^{q(x)}[0,1]}\leq Me^{-ck}$ so that  $\sum_{k=0}^{\infty}\|R(\cdot +k)\|_{L^{q(x)}[0,1]}<\infty ,$ as claimed.
\item[(ii)] By a mild solution of \eqref{left-bruka}, we mean the function $t\mapsto \int^{t}_{-\infty}R_{\gamma}(t-s)g(s)\, ds,$ $t\in {\mathbb R}.$ Let $p\in {\mathcal P}([0,1]),$ and let $\|R_{\gamma}(\cdot)\|_{L^{q(x)}[0,1]}<\infty .$ Then, for $k\geq 1,$ we have $\|R_{\gamma}(\cdot +k)\|_{L^{q(x)}[0,1]}\leq M_{2}k^{-1-\gamma}.$ Hence,  $\sum_{k=0}^{\infty}\|R_{\gamma}(\cdot +k)\|_{L^{q(x)}[0,1]}<\infty $ and we can apply Proposition \ref{ravi-and-variable}.
\end{itemize}
\end{example} 

In the following proposition, whose proof is very similar to that of \cite[Proposition 3.12]{toka-marek-prim}, we state some invariance properties of generalized asymptotical almost periodicity in Lebesgue spaces with variable exponents
$L^{p(x)}$ under the action of finite convolution products (see also \cite[Proposition 2.7.5, Lemma 2.9.3]{nova-mono} for similar results). This proposition generalizes \cite[Proposition 2.13]{EJDE} provided that $p>1$ in its formulation.

\begin{prop}\label{stepanov-almost-periodic-p(x)}
Suppose that $ p\in {\mathcal P}([0,1]),$ $ q\in {D}_{+}([0,1]),$ $1/p(x) +1/q(x)=1$
and $(R(t))_{t> 0}\subseteq L(X)$ is a strongly continuous operator family satisfying that, for every
$t\geq 0,$ we have that
$m_{t}:=\sum_{k=0}^{\infty}\|R(\cdot +t+k)\|_{L^{q(x)}[0,1]}<\infty .$
Suppose, further, that $\check{g} : {\mathbb R} \rightarrow X$ is $S^{p(x)}$-almost periodic, $q\in L_{S}^{p(x)}( [0,\infty) :X)$ and 
$f(t)=g(t)+q(t),$ $t\geq 0.$  Let $r_{1},\ r_{2}\in {\mathcal P}([0,1])$ and
the following hold:
\begin{itemize}
\item[(i)]  For every $t\geq 0$, the mapping $x\mapsto \int^{t+x}_{0}R(t+x-s) q(s)\, ds,$ $x\in [0,1]$ belongs to the space $ L^{r_{1}(x)}([0,1] : X) $ and we have
\begin{align*}
\lim_{t\rightarrow +\infty}\Biggl\| \int^{t+x}_{0}R(t+x-s) q(s)\, ds\Biggr\|_{L^{r_{1}(x)}[0,1]}=0.
\end{align*}
\item[(ii)] For every $t\geq 0,$ the mapping  $x\mapsto m_{t+x},$ $x\in [0,1]$ belongs to the space $L^{r_{2}(x)}[0,1]$ and we have
$$
\lim_{t\rightarrow +\infty}\bigl\| m_{t+x}\bigr\|_{L^{r_{2}(x)}[0,1]}=0.
$$
\end{itemize}
Then the function $H(\cdot),$ given by
\begin{align*}
H(t):=\int^{t}_{0}R(t-s)f(s)\, ds,\quad t\geq 0,
\end{align*}
is well-defined, bounded and belongs to the class $APS^{p(x)}({\mathbb R} : X)+S^{r_{1}(x)}_{0}([0,\infty):X) + S^{r_{2}(x)}_{0}([0,\infty):X),$ with the meaning clear.
\end{prop}

\begin{rem}\label{observ}
In \cite[Remark 2.14]{EJDE}, we have examined the conditions under which the function $H(\cdot)$ defined above is asymptotically almost periodic, provided that the function $g(\cdot)$ is $S^{p}$-almost periodic for some $p\in [1,\infty).$ The interested reader may try to analyze similar problems with function $\check{g}(\cdot)$ being $S^{p(x)}$-almost periodic for some $p\in {\mathcal P}([0,1]).$
\end{rem}

It is clear that Proposition \ref{stepanov-almost-periodic-p(x)} can be applied in the qualitative analysis of solutions for a wide class of inhomogeneous abstract Cauchy problems and inclusions. For example, we can simply apply Proposition \ref{stepanov-almost-periodic-p(x)} in the study of following fractional relaxation inclusion
\[
\hbox{(DFP)}_{f,\gamma} : \left\{
\begin{array}{l}
{\mathbf D}_{t}^{\gamma}u(t)\in {\mathcal A}u(t)+f(t),\ t> 0,\\
\quad u(0)=x_{0},
\end{array}
\right.
\]
where ${\mathbf D}_{t}^{\gamma}$ denotes the Caputo fractional derivative of order $\gamma \in (0,1],$ $x_{0}\in X$ and
$f : [0,\infty) \rightarrow X$ satisfies certain properties. By a classical solution of (DFP)$_{f,\gamma},$ we mean any function $u\in C([0,\infty) : X)$ satisfying that the function\index{solution!classical solution of (DFP)$_{f,\gamma}$}
${\mathbf D}_{t}^{\gamma}u(t)$ is well-defined on any finite interval $(0,T]$ and belongs to the space $C((0,T] : X),$ as well as that
$u(0)=u_{0}$ and
${\mathbf D}_{t}^{\gamma}u(t)
-f(t) \in  {\mathcal A}u(t)$ for $t>0.$ Under certain conditions \cite{nova-mono}, the classical solution of (DFP)$_{f,\gamma}$
is given by the formula
$$
u(t)=S_{\gamma}(t)x_{0}+\int^{t}_{0}R_{\gamma}(t-s)f(s)\, ds,\quad t\geq 0.
$$

Now we will present a few illustrative examples:

\begin{example}\label{illust-periodic}
...
\end{example}

\section{Conclusion and final remarks}\label{section5}

...

\end{document}